\documentclass[a4paper]{amsart}
\usepackage[latin1]{inputenc}
\usepackage{amssymb}
\usepackage{amsmath}
\usepackage{mathrsfs}
\usepackage{eufrak}
\usepackage{amsthm}
\usepackage{amsfonts}
\usepackage{textcomp}
\usepackage{graphicx}
\usepackage[pdftex]{color}
\usepackage{paralist}
\usepackage[shortlabels]{enumitem}
\usepackage{hyperref}
\usepackage{comment}
\usepackage[arrow, matrix, curve]{xy}
\usepackage{tikz-cd} 
\usetikzlibrary{arrows,intersections,shapes.misc}
\makeatletter
\tikzset{autosave background path/.code={
		\ifcsname tikz@fig@name\endcsname
		\tikzset{name path=bp\tikz@fig@name}%
		\fi}}   
\makeatother

\newtheorem*{corollary*}{Corollary}
\newtheorem{theorem}{Theorem}[section]

\newtheorem{lemma}[theorem]{Lemma}

\newtheorem{question}[theorem]{Question}

\newtheorem*{claim*}{Claim}

\theoremstyle{definition}

\newtheorem*{theorem }{Theorem}

\theoremstyle{remark}

\numberwithin{equation}{theorem}

\makeatletter
\renewcommand*\env@matrix[1][\
arraystretch]{%
  \edef\arraystretch{#1}%
  \hskip -\arraycolsep
  \let\@ifnextchar\new@ifnextchar
  \array{*\c@MaxMatrixCols c}}
\makeatother

\renewcommand{\mod}{\operatorname{mod}}

\newcommand{\Hom}{\operatorname{Hom}}

\renewcommand{\mod}{\operatorname{mod}}

\newcommand{\Tr}{\operatorname{Tr}}

\begin{document}

\title{On the extension-closed property for the subcategory $\Tr(\Omega^2(\mod-A))$}
\date{\today}


\author{Bernhard B\"ohmler}
\address{FB Mathematik, TU Kaiserslautern, Gottlieb-Daimler-Str. 48, 67653 Kaiserslautern, Germany}
\email{boehmler@mathematik.uni-kl.de}

\author{Ren\'{e} Marczinzik}
\address{Institute of algebra and number theory, University of Stuttgart, Pfaffenwaldring 57, 70569 Stuttgart, Germany}
\email{marczire@mathematik.uni-stuttgart.de}

\begin{abstract}
We present a monomial quiver algebra $A$ having the property that the subcategory~$\Tr(\Omega^2(\mod-A))$ is not extension-closed. This answers a question raised by Idun Reiten.
\end{abstract}

\maketitle
\section*{Introduction}
For a noetherian ring $A$, let $\mod-A$ denote the category of finitely generated right $A$-modules. Moreover, let $\Omega^i(\mod-A)$ denote the full subcategory of $i$-th syzygy modules of finitely generated $A$-modules.
The study of syzygies is of major importance in homologcal algebra and representation theory, we mention for example Hilbert's famous syzygy theorem \cite{H} and the work of Auslander and Reiten on the subcategories of syzygy modules for noetherian rings in \cite{AR}. Throughout, we assume that all subcategories are full, unless stated otherwise. Let $\mathcal{C}$ denote a subcategory of $\mod-A$ for a noetherian ring $A$ and let $\Tr$ denote the Auslander-Bridger transpose that gives an equivalence between the stable categories~$\underline{\mod}-A$ and $\underline{\mod}-A^{op}$, see for example \cite{AB}.
Following \cite{R}, we define $\Tr(\mathcal{C})$ to be the smallest additive subcategory of $\mod-A^{op}$ containing the modules~$\Tr(M)$ for $M \in \mathcal{C}$ and all the projective $A^{op}$-modules.
Recall that a subcategory $\mathcal{C}$ is said to be closed under extensions if for any $M, N \in \mathcal{C}$, also every module~$X$ is in $\mathcal{C}$ if there is an exact sequence of the form~$0 \rightarrow M \rightarrow X \rightarrow N \rightarrow 0.$
The extension-closedness of certain subcategories can give important information on the algebra. We mention for example Theorem 0.1 in~\cite{AR} which shows that the property that the subcategories $\Omega^i(\mod-A)$ are extension-closed for~$1 \leq i \leq n$ for an algebra $A$ is equivalent to $A$ being quasi $n$-Gorenstein in the sense of \cite{IZ}.
In \cite{R}, it is noted that even if a subcategory $\mathcal{C}$ is not closed under extensions, the subcategory $\Tr(\mathcal{C})$ can still be closed under extensions.
This is for example true for~$\mathcal{C}=\Omega^1(\mod-A)$ as was remarked in \cite{R}.
Idun Reiten states in the paragraph before Proposition~1.1 in \cite{R} that the answer to the following question is not known:
\begin{question} \label{question}
Let $A$ be a noetherian ring and $i>1$. Is $\Tr(\Omega^{i}(\mod-A))$ closed under extensions?

\end{question}
\indent In this article we give a negative answer to this question. Our main result is:
\begin{theorem} \label{maintheorem}
Let $A=KQ/I$ be the finite dimensional quiver algebra over a field $K$ where $Q$ is given by
\begin{tikzcd}
2 \arrow[r, "z", shift left=0.75ex] & 1 \arrow[loop, distance=3em, out=35, in=-35, "x"] \arrow[l, "y", shift left=0.75ex]
\end{tikzcd}
and the 
\noindent relations are given by~$I=\langle xy, yz, zx, x^3 \rangle$. Then the subcategory~$\Tr(\Omega^{2}(\mod-A))$ is not closed under extensions.

\end{theorem}

\section{The proof}
In this section we prove Theorem \ref{maintheorem}. We assume that the reader is familiar with the basics of representation theory and homological algebra of finite dimensional algebras and refer for example to \cite{ARS} and \cite{SkoYam}.
Let $A=KQ/I$ be the finite dimensional quiver algebra over a field $K$ where~$Q$ is given by~
\begin{tikzcd}
	2 \arrow[r, "z", shift left=0.75ex] & 1 \arrow[loop, distance=3em, out=35, in=-35, "x"] \arrow[l, "y", shift left=0.75ex]
\end{tikzcd}
and the relations are given by $I=\langle xy, yz, zx, x^3 \rangle$. 
Let $S_i$ denote the simple $A$-modules, $J$ the Jacobson radical of $A$ and $D=\Hom_K(-,K)$ the natural duality.
$A$ has vector space dimension 7 over the field $K$. Let $e_i$ denote the primitive idempotents of the quiver algebra $A$ corresponding to the vertices $i$ for $i=1,2$.
The indecomposable projective $A$-module $P_1=e_1 A= \langle e_1, x ,y , x^2 \rangle$ has dimension vector $[3,1]$ and the indecomposable projective $A$-modules $P_2=e_2 A= \langle e_2, z, zy \rangle $ has dimension vector $[1,2]$. We remark that $P_2$ is isomorphic to the indecomposable injective module $I_2=D(Ae_2)$.
The indecomposable injective module $I_1=D(Ae_1)$ has dimension vector $[3,1]$.
We give the Auslander-Reiten quiver of $A$ here: 
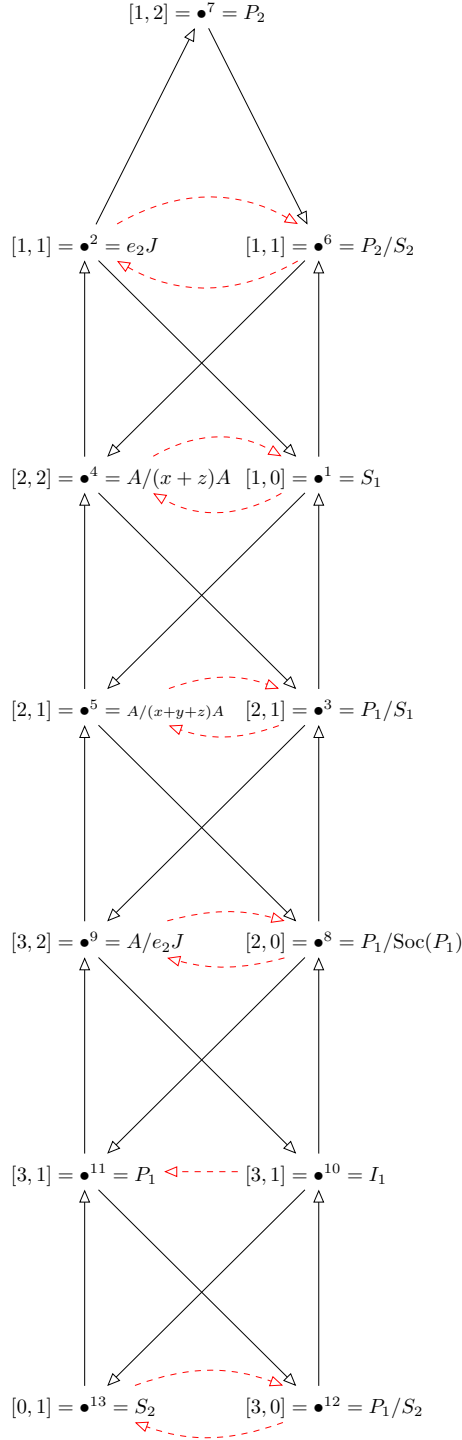
\begin{center}
\scalebox{.77}{
\begin{tikzpicture}[center bullet/.code 2 args={%
	\pgfmathsetmacro{\mydw}{0.5*(width("${}\vphantom{\bullet}#2$")-width("$#1{}$"))}%
	\tikzset{xshift=\mydw pt,yshift=0.5ex,text depth=0.25ex,text height=1em,
		autosave background path,
		execute at begin node={$#1\bullet#2$}}},
pics/connect/.style n args={4}{code={%
		\path[name path=tmp] (#2) -- (#4);
		\tikzset{name intersections={of=tmp and bp#1,by=tmp1},
			name intersections={of=tmp and bp#3,by=tmp2}}
		\draw[pic actions] (tmp1) -- (tmp2);   
}}]
\path (-2,20) coordinate (v2) (2,20) coordinate (v6) (0,24) coordinate (v7) (-2,16) coordinate (v4) (2,16) coordinate (v1) (-2,12) coordinate (v5) (2,12) coordinate (v3) (-2,8) coordinate (v9) (2,8) coordinate (v8) (-2,4) coordinate (v11) (2,4) coordinate (v10) (-2,0) coordinate (v13) (2,0) coordinate (v12);
\node[center bullet={[1,1]=}{^{2}=e_2J}] (n2) at (v2) {};
\node[center bullet={[1,1]=}{^{6}=P_2/S_2}] (n6) at (v6){};
\node[center bullet={[1,2]=}{^7=P_2}] (n7) at (v7){};
\node[center bullet={[2,2]=}{^4=A/(x+z)A}] (n4) at (v4){};
\node[center bullet={[1,0]=}{^1=S_1}] (n1) at (v1){};
\node[center bullet={[2,1]=}{^5={\scriptstyle{A/(x+y+z)A}}}] (n5) at (v5){};
\node[center bullet={[2,1]=}{^3=P_1/S_1}] (n3) at (v3){};
\node[center bullet={[2,0]=}{^8=P_1/\mbox{Soc}(P_1)}] (n8) at (v8){};
\node[center bullet={[3,2]=}{^9=A/e_2J}] (n9) at (v9){};
\node[center bullet={[3,1]=}{^{10}=I_1}] (n10) at (v10){};
\node[center bullet={[3,1]=}{^{11}=P_1}] (n11) at (v11){};
\node[center bullet={[3,0]=}{^{12}=P_1/S_2}] (n12) at (v12){};
\node[center bullet={[0,1]=}{^{13}=S_2}] (n13) at (v13){};

\path[-open triangle 45] pic{connect={n2}{v2}{n7}{v7}}
pic{connect={n2}{v2}{n1}{v1}} pic{connect={n7}{v7}{n6}{v6}}
pic{connect={n4}{v4}{n2}{v2}} pic{connect={n1}{v1}{n6}{v6}}
pic{connect={n6}{v6}{n4}{v4}} pic{connect={n5}{v5}{n4}{v4}}
pic{connect={n3}{v3}{n1}{v1}} pic{connect={n4}{v4}{n3}{v3}}
pic{connect={n1}{v1}{n5}{v5}} pic{connect={n9}{v9}{n5}{v5}}
pic{connect={n8}{v8}{n3}{v3}} pic{connect={n3}{v3}{n9}{v9}}
pic{connect={n5}{v5}{n8}{v8}} pic{connect={n11}{v11}{n9}{v9}}
pic{connect={n10}{v10}{n8}{v8}} pic{connect={n9}{v9}{n10}{v10}}
pic{connect={n8}{v8}{n11}{v11}} pic{connect={n13}{v13}{n11}{v11}}
pic{connect={n12}{v12}{n10}{v10}} pic{connect={n11}{v11}{n12}{v12}}
pic{connect={n10}{v10}{n13}{v13}};

\path[-open triangle 45] (n6) edge[bend angle=30, bend left, dashed, red] node[left] {}  (n2);
\path[-open triangle 45] (n2) edge[bend angle=30, bend left, dashed, red] node[left] {}  (n6);
\path[-open triangle 45] (n1) edge[bend angle=30, bend left, dashed, red] node[left] {}  (n4);
\path[-open triangle 45] (n4) edge[bend angle=30, bend left, dashed, red] node[left] {}  (n1);
\path[-open triangle 45] (n3) edge[bend angle=20, bend left, dashed, red] node[left] {}  (n5);
\path[-open triangle 45] (n5) edge[bend angle=20, bend left, dashed, red] node[left] {}  (n3);
\path[-open triangle 45] (n9) edge[bend angle=15, bend left, dashed, red] node[left] {}  (n8);
\path[-open triangle 45] (n8) edge[bend angle=15, bend left, dashed, red] node[left] {}  (n9);
\path[-open triangle 45] (n13) edge[bend angle=20, bend left, dashed, red] node[left] {}  (n12);
\path[-open triangle 45] (n12) edge[bend angle=20, bend left, dashed, red] node[left] {}  (n13);

\path[-open triangle 45] (n10) edge[dashed, red] node[left] {}  (n11);
\end{tikzpicture}
}
\end{center}
\noindent We remark that the algebra $A$ is a representation-finite string algebra and thus the determination of the Auslander-Reiten quiver is well known, see for example \cite{BR}. We leave the verficiation of the Auslander-Reiten quiver of $A$ to the reader. The Auslander-Reiten quiver can also be verified using \cite{QPA}, since by the theory in \cite{BR}, the Auslander-Reiten quiver does not depend on the field because $A$ is a representation-finite string algebra.\newline\newline
The dashed arrows in the Auslander-Reiten quiver indicate the Auslander-Reiten translates and the other arrows the irreducible maps.
There are $13$ indecomposable $A$-modules and we label them with their dimension vectors and an easy description via indecomposable projectives or injectives, their radicals, simple modules, or ideals in the algebra $A$.
For finite dimensional algebras we have $D \Tr=\tau$, the Auslander-Reiten translate. Since $D$ is a duality, the subcategory $\Tr (\Omega^2(\mod-A))$ is extension-closed if and only if~$\tau ( \Omega^2(\mod-A))$ is extension-closed. We will consider the subcategory $\tau ( \Omega^2(\mod-A))$ in the following.
We sometimes use the notation $\tau_i:=\tau(\Omega^{i-1})$ for the higher Auslander-Reiten translates that play an important role for example in the theory of cluster-tilting subcategories, see \cite{Iya}.
Let~$\tilde{I}:=(x+z)A$ denote the right ideal generated by $x+z$, which has vector basis basis $\{x+z, x^2, zy\}$. We remark that~$M_1:=A/\tilde{I}$ is an indecomposable $A$-module with dimension vector $[2,2]$.
Let $M_2:=P_2/S_2= e_2 A/(zy) A$. This is an indecomposable $A$-module with dimension vector $[1,1]$.

\begin{lemma}
The following assertions hold:
\begin{enumerate}
\item $\tau_3(M_1) \cong M_1$ and hence $M_1 \in \tau ( \Omega^2(\mod-A))$.
\item $\tau_3(M_2) \cong M_2$ and hence $M_2 \in \tau ( \Omega^2(\mod-A))$.
\end{enumerate}

\end{lemma}
\begin{proof}
\begin{enumerate}
\item We have the exact sequences
$$0 \rightarrow (x+z)A \rightarrow A \rightarrow A/(x+z)A \rightarrow 0$$
and
$$0 \rightarrow S_1 \rightarrow e_1 A \rightarrow (x+z)A \rightarrow 0.$$
It follows that $\Omega^2(M_1) \cong S_1$.
We can see from the Auslander-Reiten quiver that $\tau(S_1) \cong M_1$. Therefore
$\tau(\Omega^2(M_1)) \cong \tau(S_1) \cong M_1$.
\item We have the exact sequences
$$0 \rightarrow S_2 \rightarrow e_2 A \rightarrow e_2A /S_2 \rightarrow 0$$
and 
$$0 \rightarrow e_2 J^1 \rightarrow e_2 A \rightarrow S_2 \rightarrow 0.$$
It follows that $\Omega^2(M_2) \cong e_2 J$ and we can see from the Auslander-Reiten quiver that $\tau(e_2J) \cong M_2$. Therefore
$\tau(\Omega^2(M_2)) \cong \tau(e_2 J) \cong M_2$.
\end{enumerate}
\end{proof}

\noindent Now, we construct a module $W$ such that $W$ is an extension of $M_2$ by $M_1$, but $W$ does not lie in~$\tau ( \Omega^2(\mod-A))$.

\begin{lemma}
There is a short exact sequence 
$$0 \rightarrow M_1 \xrightarrow{u} W \rightarrow M_2 \rightarrow 0.$$
We have $W \cong P_2 \oplus U$, where $U=A/((x+y+z)A)$ is indecomposable.
The short exact sequence is therefore not split exact.
\end{lemma}

\newpage

\begin{proof}
We define $u=(u_1,u_2)$ and $W$ as follows using quiver representations:
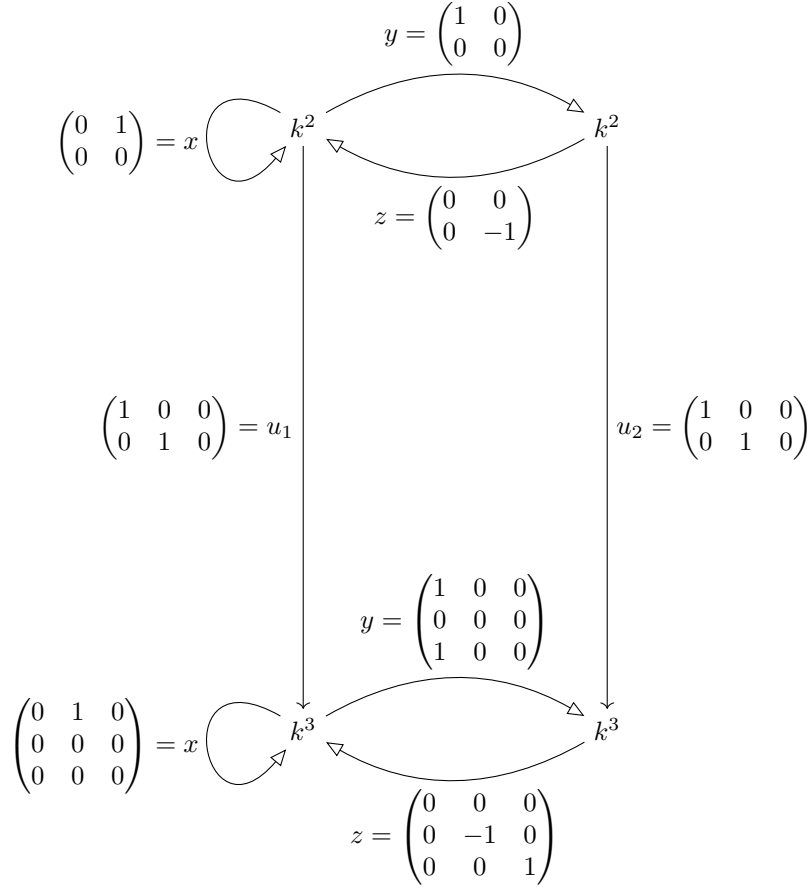
\begin{center}
\begin{tikzpicture}

\node (v1) at (-2,4) {$k^2$};
\node (v2) at (2,4) {$k^2$};
\node (v3) at (-2,-4) {$k^3$};
\node (v4) at (2,-4) {$k^3$};

\draw [-open triangle 45] (v1) to [out=150,in=230,looseness=10] node[left] {$\begin{pmatrix}[1] 0 & 1\\ 0 & 0\end{pmatrix}=x$} (v1);

\draw [->] (v1) edge node[left] {$\begin{pmatrix}[1] 1 & 0 & 0\\ 0 & 1 & 0\end{pmatrix}=u_1$}  (v3);
\draw [->] (v2) edge node[right] {$u_2=\begin{pmatrix}[1] 1 & 0 & 0\\ 0 & 1 & 0\end{pmatrix}$}  (v4);

\draw [-open triangle 45] (v1) edge[bend left] node[above] {$y=\begin{pmatrix}[1] 1 & 0\\ 0 & 0\end{pmatrix}$}  (v2);
\draw [-open triangle 45] (v2) edge[bend left] node[below] {$z=\begin{pmatrix}[1] 0 & 0\\ 0 & -1\end{pmatrix}$}  (v1);

\draw [-open triangle 45] (v3) edge[bend left] node[above] {$y=\begin{pmatrix}[1] 1 & 0 & 0\\ 0 & 0 & 0\\ 1 & 0 & 0\end{pmatrix}$}  (v4);
\draw [-open triangle 45] (v4) edge[bend left] node[below] {$z=\begin{pmatrix}[1] 0 & 0 & 0\\ 0 & -1 & 0\\ 0 & 0 & 1\end{pmatrix}$}  (v3);

\draw [-open triangle 45] (v3) to [out=150,in=230,looseness=10] node[left] {$\begin{pmatrix}[1] 0 & 1 & 0\\ 0 & 0 & 0\\ 0 & 0 & 0\end{pmatrix}=x$} (v3);

\end{tikzpicture}
\end{center}
It is then easy to see that $u$ is a monomorphism whose domain is isomorphic to $M_1$, and that the cokernel of $u$ is isomorphic to $M_2$.
The quiver representation of $P_2$ looks as follows:
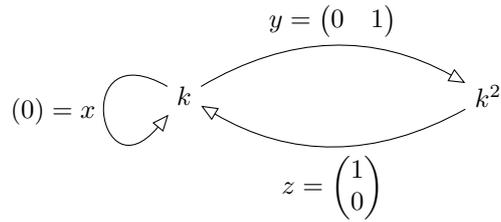
\begin{center}
\begin{tikzpicture}

\node (v1) at (-2,4) {$k$};
\node (v2) at (2,4) {$k^2$};

\draw [-open triangle 45] (v1) to [out=150,in=230,looseness=10] node[left] {$(0)=x$} (v1);

\draw [-open triangle 45] (v1) edge[bend left] node[above] {$y=\begin{pmatrix}[1] 0 & 1\end{pmatrix}$}  (v2);
\draw [-open triangle 45] (v2) edge[bend left] node[below] {$z=\begin{pmatrix}[1] 1\\ 0\end{pmatrix}$}  (v1);

\end{tikzpicture}
\end{center}

The quiver representation of $U$ looks as follows:
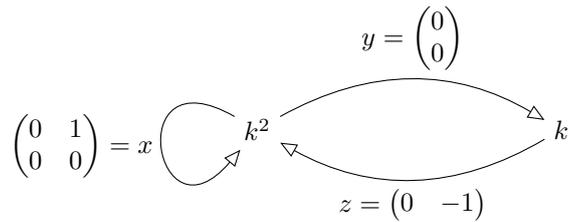
\begin{center}
\begin{tikzpicture}

\node (v1) at (-2,4) {$k^2$};
\node (v2) at (2,4) {$k$};

\draw [-open triangle 45] (v1) to [out=150,in=230,looseness=10] node[left] {$\begin{pmatrix}[1] 0 & 1\\ 0 & 0\end{pmatrix}=x$} (v1);

\draw [-open triangle 45] (v1) edge[bend left] node[above] {$y=\begin{pmatrix}[1] 0\\ 0\end{pmatrix}$}  (v2);
\draw [-open triangle 45] (v2) edge[bend left] node[below] {$z=\begin{pmatrix}[1] 0 & -1\end{pmatrix}$}  (v1);

\end{tikzpicture}
\end{center}

Note that the module $U$ appears as the vertex with number $5$ in the Auslander-Reiten quiver of $A$ and is indecomposable.
A direct verification that we leave to the reader shows that $W \cong P_2 \oplus U$. 
This proves that the short exact sequence is not split exact, since $W$ is not isomorphic to the direct sum of the indecomposable modules $M_1$ and $M_2$.
\end{proof}

\begin{lemma}
The module $U=A/(x+y+z)A$ does not lie in $\tau ( \Omega^2(\mod-A))$.
\end{lemma}
\begin{proof}
We prove the equivalent statement that $\tau^{-1}(U)$ is not a second syzygy module.
From the Auslander-Reiten quiver we can see that $\tau^{-1}(U) = P_1/S_1$.
Now we use the result that a general module~$X$ over an Artin algebra is a direct summand of an $n$-th syzygy module if and only if $X$ is a direct summand of $P' \oplus \Omega^n(\Omega^{-n}(X))$ for some projective module $P'$, see for example Proposition 3.2 of \cite{AS}.
The two short exact sequences
$$0 \rightarrow P_1 /S_1 \rightarrow D(A) \rightarrow S_2 \oplus U \rightarrow 0$$
and
$$0 \rightarrow S_2 \oplus U \rightarrow D(A) \rightarrow S_1 \oplus M_2 \rightarrow 0$$
\noindent give rise to minimal injective coresolutions. Now $\Omega^2(S_1) \cong S_1 \oplus S_2 \oplus e_2J^1$ and $\Omega^2(M_2) \cong e_2 J^1$.\newline
But the indecomposable module $P_1 / S_1$ is not a direct summand of a module of the form
$$P' \oplus \Omega^2(\Omega^{-2}(P_1/S_1))=P' \oplus S_1 \oplus S_2 \oplus e_2J^1 \oplus e_2 J^1$$
\noindent for any projective module $P'$. Hence, it is not a $2^{\text{nd}}$ syzygy module.
\end{proof}

The subcategory $\tau ( \Omega^2(\mod-A))$ and equivalently the subcategory $\Tr ( \Omega^2(\mod-A))$ is therefore not extension-closed.

\section{QPA calculation}
In this section we illustrate how to verify that our result is correct over the field with three elements using QPA. You can copy and paste the following code into GAP.

\begin{tiny}
\begin{verbatim}
LoadPackage("qpa");
Q:=Quiver(2,[[1,1,"x"],[1,2,"y"],[2,1,"z"]]);kQ:=PathAlgebra(GF(3),Q);AssignGeneratorVariables(kQ);rel:=[x*y,y*z,z*x,x^3];A:=kQ/rel;
projA:=IndecProjectiveModules(A);
P1:=projA[1];P2:=projA[2];
M2:=CoKernel(SocleOfModuleInclusion(P2));
I_neu:=RightIdeal(A, [A.x+A.z]);O:=RightAlgebraModuleToPathAlgebraMatModule(RightAlgebraModule(A, \*, I_neu));
M1:=TransposeOfModule(NthSyzygy(TransposeOfModule(O),1));
IsomorphicModules(M1,DTr(NthSyzygy(M1,2)));
IsomorphicModules(M2,DTr(NthSyzygy(M2,2)));
IsIndecomposableModule(M1);
IsIndecomposableModule(M2);
Size(ExtOverAlgebra(M2,M1)[2]);
ext := ExtOverAlgebra(M2,M1);
maps := ext[2];
monos := List( maps, h -> PushOut( ext[1], h )[ 1 ] );
u:=monos[1];
IsInjective(u);
IsomorphicModules(Source(u),M1);
IsomorphicModules(CoKernel(u),M2);
W:=Range(u);
WW:=DecomposeModule(W);
OO:=WW[1];IsomorphicModules(OO,P2);
U:=WW[2];
IsIndecomposableModule(U);
K:=TrD(U);
IsNthSyzygy(K,2);
\end{verbatim}
\end{tiny}
Here we constructed an injective map $u: M_1 \rightarrow W$, where the modules $M_1$ and $M_2:=\text{Coker}(u)$ are indecomposable and satisfy $M_1 \cong \tau_3(M_1)$ and $M_2 \cong \tau_3(M_2)$.
It is verified that the module $W$ is isomorphic to $P_2 \oplus U$, where $P_2$ is the indecomposable projective module corresponding to the second vertex and $U$ is some indecomposable module that is not in $\tau(\Omega^2(\mod-A))$, since $\tau^{-1}(U)$ is not a second syzygy module.
The subcategory $\tau(\Omega^2(\mod-A))$ is therefore not extension-closed.

\section*{Acknowledgements} 
Bernhard B\"ohmler gratefully acknowledges funding by the DFG (SFB/TRR 195). Ren{\'e} Marczinzik gratefully acknowledges funding by the DFG (with project number 428999796). We profited from the use of the GAP-package \cite{QPA}.

\end{document}